\documentclass{amsart}
\usepackage{tabu}
\usepackage{amssymb} 
\usepackage{amsmath} 
\usepackage{amscd}
\usepackage{amsbsy}
\usepackage{comment, enumerate}
\usepackage[matrix,arrow]{xy}
\usepackage{hyperref}
\usepackage{mathrsfs}
\usepackage{color}
\usepackage{mathtools,caption}

\usepackage{algorithm}
\usepackage{algpseudocode}
\renewcommand{\algorithmicrequire}{\textbf{Input:}}
\renewcommand{\algorithmicensure}{\textbf{Output:}}

\newtheorem{lemma}{Lemma}[section]
\newtheorem{proposition}{Proposition}[section]

\newtheorem{theorem}{Theorem}[section]
\newtheorem{definition}{Definition}[section]

\newcommand{\QQ}{\mathbb Q}
\newcommand{\ZZ}{\mathbb Z}
\newcommand{\OK}{\mathcal O_K}
\newcommand{\OL}{\mathcal O_L}

\newcommand{\tu}{\tilde u}
\newcommand{\fp}{\mathfrak p}
\newcommand{\fq}{\mathfrak q}
\newcommand{\fz}{\mathfrak z}
\newcommand{\F}{\mathbb F}
\newcommand{\Bp}{10^4}

\DeclareMathOperator{\lcm}{lcm}

\begin{document}

\title[]{On the solutions of the Diophantine equation $(x-d)^2+x^2+(x+d)^2=y^n$ for $d$ a prime power}

\author{Angelos Koutsianas}
\address{Department of Mathematics, The University of British Columbia, 1984 Mathematics Road, Vancouver, BC, V6T 1Z2, Canada}
\email{akoutsianas@math.ubc.ca}

\date{\today}

\keywords{Exponential equation, Lehmer sequences, primitive divisors}
\subjclass[2010]{Primary 11D61}

\begin{abstract}
In this paper, we determine the primitive solutions of the Diophantine equation $(x-d)^2+x^2+(x+d)^2=y^n$ when $n\geq 2$ and $d=p^b$, $p$ a prime and $p\leq 10^4$. The main ingredients are the characterization of primitive divisors on Lehmer sequences and the development of an algorithmic method of proving the non-existence of integer solutions of the equation $f(x)=a^b$, where $f(x)\in\ZZ[x]$, $a$ a positive integer and $b$ an arbitrary positive integer.
\end{abstract}

\maketitle

\section{Introduction}

The question when a sum of consecutive powers is a perfect power has a long and rich history. In $1875$ Lucas \cite{Lucas75} asks for the integral solutions of the equation
\begin{equation}
1^2 + 2^2 +\cdots + x^2 = y^2,
\end{equation}
and is Watson who gives the first satisfying solution \cite{Watson18}. In 1956, Sch\"{a}ffer \cite{Schaffer56} generalises Lucas question and studies the equation
\begin{equation}\label{eq:Schaffer}
1^k + 2^k + \cdots +x^k = y^n.
\end{equation}
Sch\"{a}ffer proves that \eqref{eq:Schaffer} has only finitely many solutions for fix $k$ and $n$ except for a finite number of cases which he determines. In \cite{BennettGyoryPinter04} the authors complete solve \eqref{eq:Schaffer} for $k\leq 11$ and the case $n$ is even, $k$ odd and $k\leq 170$ is studied in \cite{Pinter07}. 

The last few years many mathematicians have focused on the more general equation
\begin{equation}\label{eq:general}
x^k + (x + d)^k + \cdots + (x + (r-1)d)^k = y^n,\quad x,y,d,r,k,n\in\ZZ, n\geq 2,
\end{equation}
and many specific cases have been studied. For example, the case $k=3$, $d=1$ and $r\leq 50$ is considered in  \cite{BennettPatelSiksek17} and for $k=2$, $d=1$ and $r\leq 10$ in \cite{Patel17}. Patel and Siksek \cite{PatelSiksek17} prove that \eqref{eq:general} has no solutions for $k$ being even for almost all $r\geq 2$. We also refer to \cite{BaiZhang13,Soydan17,BerczesPinkSavasSoydan18} for more cases.

Among the many different cases of \eqref{eq:general} the case of three powers
\begin{equation}\label{eq:three_terms}
(x-d)^k+x^k+(x+d)^k=y^n,
\end{equation} 
has attracted a lot of attention. Equation \eqref{eq:three_terms} has been studied for small values of $k$ and $d=1$ in \cite{BennettPatelSiksek16,Zhang14} and for $d>1$ in \cite{Zhang17,ArgaezPatel}. In this paper we study the equation
\begin{equation}\label{eq:main}
	(x-d)^2+x^2+(x+d)^2=y^n.
\end{equation}
In \cite{KoutsianasPatel18} the authors determine all non-trivial primitive solutions of \eqref{eq:main} for $d\leq 10^4$. A very natural and interesting question is to study equation (\ref{eq:main}) for an infinitely family of $d$. Very recently Zhang \cite{Zhang17} proved that (\ref{eq:three_terms}) has no solutions for $k=4$ and $n\geq 11$ when $d$ lies in a suitable infinitely family. In this paper we study \eqref{eq:main} when $d$ is a prime power. A solution of \eqref{eq:main} is called \textit{primitive} if $(x,y)=1$. Moreover, a solution is called \textit{non-trivial} if $x\neq 0$.

\begin{theorem}\label{th:main_gcd_one}
Let $n\geq 2$ be an integer. The non-trivial primitive solutions of \eqref{eq:main} where $d=p^b$ with $b\geq 0$, $p$ a prime and $p\leq\Bp$ are the ones in Table \ref{table:solutions}.
\end{theorem}

The main ingredient of the proof of Theorem \ref{th:main_gcd_one} is the classification of primitive divisors  of Lehmer sequence \cite{BiluHanrotVoutier01}. For fix $d$ this approach has already been used in \cite{KoutsianasPatel18} and in this paper we extend it when $d$ is an arbitrary prime power.

The paper is organized as follows. In Section \ref{sec:Lehmer_preliminaries} we recall the main definitions and terminology about Lehmer sequences. In Section \ref{sec:non_solution_f_an} we describe an algorithmic method of proving the non-existence of integer solutions to the equation $f(x)=a^n$ for given $f(x)\in\ZZ[x]$ and integer $a>1$. In Section \ref{sec:Lehmer_seq} we construct a Lehmer sequence from a solution of \eqref{eq:main} and we study some basic properties of the sequence. Finally, in Section \ref{sec:proof} we give the proof of Theorem \ref{th:main_gcd_one}.

We have used the mathematical software package \texttt{Sage} \cite{Sage} for the computations of this paper. The code can be found at \textit{https://sites.google.com/site/angelos\\koutsianas/research}.

\begin{center}
\begin{tabular}{| c | c |}
	\hline
	$ p $ & $(|x|,y,b,n)$\\
	\hline
	$ 2 $ & $(21,11,1,3)$\\
	\hline
	$ 7 $ & $ (3, 5, 1, 3)$\\
	\hline
	$ 79 $ & $ (63, 29, 1, 3)$\\
	\hline
	$ 223 $ & $ (345, 77, 1, 3)$\\
	\hline
	$ 439 $ & $ (987, 149, 1, 3)$\\
	\hline
	$ 727 $ & $ (2133, 245, 1, 3)$\\
	\hline
	$ 1087 $ & $ (3927, 365, 1, 3)$\\
	\hline
	$ 3109 $ & $ (627, 29, 1, 5)$\\
	\hline
	$ 3967 $ & $ (27657, 1325, 1, 3)$\\
	\hline
	$ 4759 $ & $ (36363, 1589, 1, 3)$\\
	\hline
	$5623$ & $(46725,1877,1,3)$\\
	\hline
	$8647$ & $(89187, 2885, 1, 3)$\\
	\hline
\end{tabular}
\captionof{table}{Non-trivial primitive solutions $(|x|,y,b,n)$.}\label{table:solutions}
\end{center}

\section{Lehmer sequenses}\label{sec:Lehmer_preliminaries}

The characterization of primitive divisors of Lehmer sequences in \cite{BiluHanrotVoutier01} is used in the proof of Theorem \ref{th:main_gcd_one}. In this section we recall the main definitions and terminology about Lehmer sequences and we recommend \cite{BiluHanrotVoutier01} for a more detailed exposition.

Let $\alpha,\beta$ be two algebraic integers such that $(\alpha +\beta)^2$ and $\alpha\beta$ are non-zero coprime rational integers and $\alpha/\beta$ is not a root of unity. Then, the pair $(\alpha,\beta)$ is called a \textit{Lehmer pair}. For a given Lehmer pair $(\alpha,\beta)$ we define the corresponding \textit{Lehmer sequence} given by
\begin{equation}
\tu_n=\tu_n(\alpha,\beta)=
\begin{cases}
\frac{\alpha^n-\beta^n}{\alpha-\beta}, & n\text{ odd},\\
\frac{\alpha^n-\beta^n}{\alpha^2-\beta^2}, & n\text{ even}.
\end{cases}
\end{equation}

\begin{definition}
Let $(\alpha,\beta)$ be a Lehmer pair. A prime number $p$ is called \textit{primitive divisor} of $\tu_n$ if $p$ divides $\tu_n$ but does not divide $(\alpha^2-\beta^2)^2\cdot\tu_1\cdots\tu_{n-1}$. 

In case $\tu_n$ has no primitive divisors then the pair $(\alpha,\beta)$ is called \textit{$n$-defective Lehmer pair}. We say that an integer $n$ is \textit{totally non-defective} if no Lehmer pair is $n$-defective.   
\end{definition}

\begin{theorem}[\cite{BiluHanrotVoutier01}]\label{thm:non_defective}
Every integer $n>30$ is totally non-defective. Moreover, if $n>13$ is a prime then it is is totally non-defective.
\end{theorem}

\begin{definition}
Two Lehmer pairs $(\alpha_1,\beta_1)$ and $(\alpha_2,\beta_2)$ are equivalent if $\alpha_1/\alpha_2=\beta_1/\beta_2\in\{\pm 1,\pm\sqrt{-1}\}$.
\end{definition}

For the integers $1\leq n\leq 30$ the $n$-defective Lehmer pairs are completely determined (up to equivalence) in \cite{Voutier95} (see \cite[Theorem C]{BiluHanrotVoutier01}) and \cite[Theorem 1.3]{BiluHanrotVoutier01}.

\section{No solutions of the equation $f(x)=a^n$}\label{sec:non_solution_f_an}

Let $a\neq 1$ be a positive integer and $f(x)\in\ZZ[x]$. We assume that the $\gcd$ of the coefficients of $f(x)$ divides $a$. We describe an algorithmic method which (when it succeeds) proves that the equation
\begin{equation}\label{eq:f_an}
f(x) = a^n,
\end{equation}
with $n\geq 0$ has no integral solutions.

Let $s$ be an integer coprime to $a$ and $t_s$ the order of $a$ at $\left(\ZZ/s\ZZ\right)^*$. We define
\begin{equation}
W_{s,a}(f) := \{a^k\pmod s:k=1,\cdots, t_s-1\}\bigcap\{f(i)\pmod s:i\in[1,s]\}.
\end{equation}
If $W_{s,a}(f)=\emptyset$ then we can conclude that $t_s\mid n$. Using many different $s$ and taking lcm of the $t_s$ we can find a number $t$ such that if there exists a solutions of \eqref{eq:f_an} then we have $t\mid n$. Then we check if there exists an integer $\ell$ such that $a^t\equiv 1\pmod \ell$ but $f(u)\not\equiv 1\pmod \ell$ for all $u\in\ZZ$. If such an integer $\ell$ exists the equation \eqref{eq:f_an} has no solutions (see Algorithm \ref{alg:non_solution_f_an}). 

\begin{algorithm}
	\caption{No solutions for $f(x)=a^n$}\label{alg:non_solution_f_an}
	\hspace{-2.3cm}{\algorithmicrequire} A polynomial $f(x)\in\ZZ[x]$ and a positive integer $a\neq 1$.\\
	\hspace{-0.5cm}{\algorithmicensure} \textit{True} if there are no solutions $(x,n)$ of $f(x)=a^n$, otherwise \textit{Fail}.
	\vspace{0.2cm}
	\begin{algorithmic}[1]
		\State If the $\gcd$ of the coefficients of $f(x)$ does not divide $a$, \textbf{return} 
			\textit{Fail}.
		\State Choose a finite set $S$ of positive integers $s$ coprime to $a$.
		\State $t\leftarrow 1$.
		\For{each $s\in S$}
			\State Let $t_s$ be the order of $a$ at $\left(\ZZ/s\ZZ\right)^*$.
			\State $W_{s,a}(f) \leftarrow\{a^k\pmod s:k=1,\cdots, t_s-1\}\bigcap\{f(i)\pmod s:i\in[1,s]\}$.
			\If{$W_{s,a}(f)=\emptyset$}
				\State $t\leftarrow \lcm(t,t_s)$			
			\EndIf
		\EndFor
		\State Choose a finite set $L$ of positive integers $\ell$ such that $a^t\equiv 1\pmod\ell$.
		\For{each $\ell\in L$}
			\If{$f(u)\not\equiv 1\pmod \ell$ for all $u\in(\ZZ/\ell\ZZ)$}
				\State \textbf{return} \textit{True}
			\EndIf
		\EndFor
		\State \textbf{return} \textit{Fail}.
	\end{algorithmic}
\end{algorithm}

\section{A Lehmer sequence from a primitive solution of \eqref{eq:main}}\label{sec:Lehmer_seq}

In this section we associate a Lehmer sequence to a primitive solution $(x,y)$ of \eqref{eq:main}. This construction has already been described in \cite[Section 3]{KoutsianasPatel18} and we repeat it for completeness and the convenience of the reader.

Suppose $n\geq 2$. We can rewrite \eqref{eq:main} as
\begin{equation}\label{eq:main_simply}
3x^2 + 2d^2 = y^n.
\end{equation}
Let $(x,y)$ be a non-trivial primitive solution of \eqref{eq:main_simply}. This implies that $x,y,d$ are pairwise coprime, $3\nmid dy$ and $2\nmid xy$.

\begin{lemma}\label{lem:primitive_n_even}
There are no non-trivial primitive solutions of \eqref{eq:main} for $n$ even.
\end{lemma}

\begin{proof}
Taking equation \eqref{eq:main_simply} $\pmod 3$ we have the conclusion.
\end{proof}

For the rest of this section we assume that $n$ is an odd prime. Let $K = \QQ(\sqrt{-6})$ and write $\OK=\ZZ[\sqrt{-6}]$ for its ring of integers. The field $K$ has class group isomorphic to $(\ZZ/2\ZZ)$. We rewrite equation \eqref{eq:main_simply} as 
\begin{equation}\label{eq:3squares}
(3x)^2 + 6d^2 = 3y^n.
\end{equation}
We factorise the left-hand side of equation~\eqref{eq:3squares} as
\begin{equation}
(3x + d\sqrt{-6})(3x - d\sqrt{-6}) = 3y^n.
\end{equation}
from which we obtain the ideal equation
\begin{equation}
\langle 3x + d\sqrt{-6}\rangle\langle 3x - d\sqrt{-6}\rangle = \fp_3^2\langle y\rangle^n,
\end{equation}
where $\fp_3$ is the unique prime of $\OK$ above $3$. Because $x,y,d$ are pairwise coprime and $3\nmid dy$ it follows that
\begin{equation}\label{eq:ideal_factorization}
\langle 3x + d\sqrt{-6}\rangle = \fp_3 \cdot \fz^n,
\end{equation}
where $\fz$ is an ideal of $\OK$. The ideal $\fp_3$ is not principal, hence $\fz$ is not either, and because $\fp_3^2=\langle 3\rangle$ we have
\begin{equation}
\langle 3x + d\sqrt{-6}\rangle = \fp_3^{1-n} \cdot (\fp_3 \fz)^n
=\langle 3\rangle^{(1-n)/2} (\fp_3 \fz)^n.
\end{equation}
Because the class number of $K$ is two and $\fp_3$, $\fz$ are not principal it holds $\fp_3\fz=\langle \gamma\rangle$ where $\gamma=u^\prime+v^\prime\sqrt{-6} \in \OK$ with $u^\prime,v^\prime\in\ZZ$. We can easily prove that $3\mid u^\prime$. Thus there are $u,v\in\ZZ$ such that $\gamma=3u+v\sqrt{-6}$. After possibly changing the sign of $\gamma$ we obtain
\begin{equation}\label{eq:x_r_gamma}
3x+d\sqrt{-6}=\frac{\gamma^n}{3^{(n-1)/2}}.
\end{equation}
Subtracting the conjugate equation from this equation, we obtain 
\begin{equation}
\frac{\gamma^n}{3^{(n-1)/2}} - \frac{\bar{\gamma}^n}{3^{(n-1)/2}} = 2d\sqrt{-6},
\end{equation}
or equivalently,
\begin{equation}\label{eq:Lehmer_sequence}
\frac{\gamma^n}{3^{n/2}} - \frac{\bar{\gamma}^n}{3^{n/2}} = 2d\sqrt{-2},
\end{equation}
where $\bar\gamma$ is the conjugate of $\gamma$.

Let $L = \QQ(\sqrt{-6}, \sqrt{3}) =  \QQ(\sqrt{-2}, \sqrt{3}) $. Write $\OL$ for the ring of integers of $L$ and let
\begin{equation}
\alpha = \frac{\gamma}{\sqrt{3}} \qquad \text{and} \qquad 
\beta = \frac{\bar{\gamma}}{\sqrt{3}}.
\end{equation}

\begin{lemma}\label{lem:lehmer}
Let $\alpha, \beta$ be as above. Then, $\alpha$ and $\beta$ are algebraic integers. Moreover, $(\alpha+\beta)^2$ and $\alpha\beta$ are non-zero coprime rational integers and $\alpha/\beta$ is not a root of unity. 
\end{lemma}

\begin{proof}
Let $\gamma=3u+v\sqrt{-6}$ be as above with $u,v\in\ZZ$. Then
$$
(\alpha+\beta)^2=12u^2 \qquad \text{and} \qquad \alpha\beta = 3u^2 + 2v^2.
$$
So, $(\alpha+\beta)^2$ and $\alpha\beta$ are rational integers. If $(\alpha+\beta)^2=0$ then we have $u=0$. However, from \eqref{eq:Lehmer_sequence} and the fact that $n$ is odd we understand that this can not happen. Clearly, $\alpha\beta=3u^2 + 2v^2$ is a non-zero rational integer.

We have to check that $(\alpha+\beta)^2$ and $\alpha\beta$ are coprime. Suppose they are not coprime. Then there exists a prime $\fq$ of $\OL$ dividing both. This implies that $\fq$ divides $\alpha,\beta$ and from \eqref{eq:Lehmer_sequence} we understand that $\fq$ divides $\langle 2d\sqrt{-2}\rangle$. Moreover, from \eqref{eq:x_r_gamma} we have
\begin{equation}\label{eq:y_wrt_u_v}
(\alpha\beta)^n=3x^2+2d^2=y^n,
\end{equation}
from which we conclude that $\fq$ divides $\langle y\rangle$. However, this is a contradiction to the fact that $y,d$ are coprime and $2\nmid y$.

Finally, we need to show that $\alpha/\beta=\gamma/\bar\gamma\in\OK$ is not a root of unity. Suppose $\gamma/\bar\gamma$ is a root of unity. Since the only roots of unity in $K$ are $\pm 1$ we conclude $\gamma=\pm\bar\gamma$. Then, either $v=0$ or $u=0$ which both can not hold because of \eqref{eq:Lehmer_sequence}.
\end{proof}

From Lemma \ref{lem:lehmer} the pair $(\alpha,\beta)$ is a Lehmer pair. Because $(\alpha + \beta)^2$ and $\alpha\beta$ are coprime integers that means $\gcd(12u^2,3u^2 + 2v^2)=1$, and hence $\gcd(3u,2v)=1$. Moreover, from \eqref{eq:y_wrt_u_v} we have that $y = 3u^2+2v^2$. We denote by $\{\tu_k\}$ the associate Lehmer sequence to the Lehmer pair $(\alpha,\beta)$. Substituting into equation \eqref{eq:Lehmer_sequence} we have
\begin{equation}
\left(\frac{\alpha-\beta}{2\sqrt{-2}} \right)\left(\frac{\alpha^n - \beta^n}{\alpha-\beta}\right) = d.
\end{equation}

Hence, we get
\begin{equation}
\frac{\alpha^n - \beta^n}{\alpha-\beta} = d /v,
\end{equation}
from which we understand that $v\mid d$. 
We define
\begin{equation}\label{eq:fk}
f_n(x,y) = \sum_{i=0}^{ (n-1)/2}\binom{n}{2i+1}(-2)^i 3^{\frac{n-1}{2}-i}x^{n-1-2i}y^{2i}\in\ZZ[x,y].
\end{equation}

After an elementary calculation we can see that
\begin{equation}
\tu_n = f_n(u,v).
\end{equation}

\section{Proof of Theorem \ref{th:main_gcd_one}}\label{sec:proof}

In this section we give the proof of Theorem \ref{th:main_gcd_one}. Let $(x,y)$ be a primitive non-trivial solution of \eqref{eq:main} for $n\geq 2$ and $d=p^b$ where $p$ is a prime. As we have already mentioned this implies that $x,y,d$ are pairwise coprime, so we have $p\neq 3$. Because of Lemma \ref{lem:primitive_n_even} there are no solutions for $n$ even, hence we can assume that $n$ is an odd prime. Let $K=\QQ(\sqrt{-6})$. We recall from Section \ref{sec:Lehmer_seq} that there exists $\gamma=3u+v\sqrt{-6}\in\OK$ with $u,v\in\ZZ$ such that
\begin{equation}
3x+p^b\sqrt{-6}=\frac{\gamma^n}{3^{(n-1)/2}}.
\end{equation}
The elements 
\begin{equation}
\alpha = \frac{\gamma}{\sqrt{3}} \qquad \text{and} \qquad 
\beta = \frac{\bar{\gamma}}{\sqrt{3}}.
\end{equation}
define a Lehmer sequence $\{\tu_k\}$ (see Lemma \ref{lem:lehmer}). It holds
\begin{equation}\label{eq:bprime}
\tu_n=\frac{\alpha^n - \beta^n}{\alpha-\beta}=  p^b /v.
\end{equation} 
From the last equation we conclude that $v=\pm p^t$ for $0\leq t\leq b$. We also recall that $\gcd(3u,2v)=1$ from which we conclude that $p\nmid u$.

We split the proof in cases according to the values of $b$ and $v$. The case $b=0$ is a consequence of the following result due to Nagell \cite{Nagell55}.

\begin{lemma}\label{lem:Nagell}
Let $D\geq 3$ be an odd number. Then the equation
\begin{equation}
	2+Dx^2=y^n,n>2,
\end{equation}
has no integer solutions $(x,y,n)$ with $n\nmid h(-2D)$ where $h(-2D)$ is the class number of $\QQ(\sqrt{-2D})$.
\end{lemma}

In our case we have $D=3$ and $h(-6)=2$, hence there are no solutions of \eqref{eq:main} for $b=0$. For the rest of the proof we assume that $b\geq 1$. We need the following lemma.

\begin{lemma}\label{lem:defective_pairs}
There are no $n$-defective pairs for the Lehmer pair $(\alpha,\beta)$ where $\alpha=\frac{\gamma}{\sqrt{3}}$ and $\beta=\frac{\bar\gamma}{\sqrt{3}}$ with $n\leq 30$ an odd prime and $v=p^t$ with $0<t\leq b$ unless $(u,v,p,b,n)=(\pm 1,\pm 2,2,1,3)$.
\end{lemma}

\begin{proof}
Because $\alpha=\frac{\gamma}{\sqrt{3}}=\sqrt{3}u+v\sqrt{-2}$ and $\beta=\sqrt{3}u-v\sqrt{-2}$ we have $(\alpha+\beta)^2 = 12u^2$ and $(\alpha-\beta)^2=-8v^2$. From the definition of equivalence Lehmer pairs we understand that the pair $(\pm 12u^2,\mp 8v^2)$ has to be in Table $2$ or Table $4$ in \cite{BiluHanrotVoutier01} (see also Table $2$ in \cite{Voutier95}). This never happens for $n>5$.

\textbf{Case $n=3$:} We recall that $f_3(u,v)=9u^2-2v^2$. Because $v=p^t$ and $t>0$ if $t\neq b$ we have $p\mid u$ which is a contradiction. Suppose $t=b$ then we end up with the equation $9u^2-2p^{2b}=\pm 1$. For $p\neq 2$ we can prove that there is no integer solutions by taking the equation $\pmod 4$ or $\pmod 3$. For $p=2$ we have
\begin{equation}
(3u-1)(3u+1) = 2^{2b+1}\qquad \text{or} \qquad 9u^2 + 1=2^{2b+1}.
\end{equation}
We can prove that the latter equation has no solutions by taking the equation $\pmod 4$. For the first equation we can assume that $u\equiv 1\pmod 4$. Thus $3u-1=2$ which means $u=1$. Hence, we conclude that $b=1$ and $v=\pm 2$.

\textbf{Case $n=5$:} We recall that $f_5(u,v)=45u^4-60u^2v^2 + 4v^4=5(3u^2 - 2v^2)^2 - 16v^4$. Similar to the case $n=3$ for $t>0$ and $p\neq 5$ if $t\neq b$ we have $p\mid u$ which is a contradiction. For $t=b$ and $p\neq 5$ we end up with the equation $5(3u^2 - 2p^{2b})^2 - 16p^{4b}=\pm 1$. We can prove that there is no integer solution by taking the equation $\pmod 5$ or $\pmod 4$.

Finally, we have the case $p=5$. As above for $b-t\geq 2$ we can prove that $5\mid u$ which is a contradiction. We consider the case $b = t+1$. Since $t>0$ we have $b-1>0$ and so 
\begin{equation}
5(3u^2 - 2\cdot 5^{2(b-1)})^2 - 16\cdot 5^{4(b-1)}=\pm 5.
\end{equation}
Taking the equation $\pmod{25}$ we get $15u^4\equiv \pm 5\pmod{25}$ and we conclude $3u^4\equiv \pm 1\pmod 5$ which is a contradiction. For the case $b=t$ we have 
\begin{equation}
5(3u^2 - 2\cdot 5^{2b})^2 - 16\cdot 5^{4b}=\pm 1
\end{equation} 
It is enough to take the equation $\pmod 5$ to get a contradiction.
\end{proof}

\begin{proposition}
We continue with the above notation and assumptions. There are no primitive non-trivial solutions $(x,y)$ of \eqref{eq:main} with $n$ an odd prime and $t>0$ unless $(|x|,y,p,b,n)=(21,11,2,1,3)$.
\end{proposition}

\begin{proof}
Let $(\alpha,\beta)$ be the corresponding Lehmer pair to a solution $(x,y)$. It holds $(\alpha^2 - \beta^2)^2=-96u^2v^2$. Because $t>0$ we understand that $p\mid v$ which implies $p\mid (\alpha^2 - \beta^2)^2$. Thus the Lehmer pair $(\alpha,\beta)$ is $n$-defective and according to Lemma \ref{lem:defective_pairs} this holds only for $(u,v,p,b,n)=(\pm 1,\pm 2,2,1,3)$ which corresponds to the solution given above.
\end{proof}

To finish the proof of Theorem \ref{th:main_gcd_one} we have to consider the case $v=\pm 1$. Similar to \cite[Lemma 3.2]{KoutsianasPatel18} we are able to bound $n$.

\begin{lemma}\label{lem:B}
Suppose $p\neq 3$ be a prime. Let
$$
B=\begin{cases}
p-1, & \text{if } \left(\frac{-6}{p}\right)=1,\\
p+1, & \text{if } \left(\frac{-6}{p}\right)=-1.
\end{cases}
$$
Let
$$
B_p:=\max\left(13,B\right).
$$
Then $n\leq B_p$.
\end{lemma}

\begin{proof}
Recall that the exponent $n$ is an odd prime. Suppose $n>13$. By the theorem of Bilu, Hanrot and Voutier, $\tu_n=(\alpha^n-\beta^n)/(\alpha-\beta)=\pm p^b$ is divisible by $p$ while $p$ divides neither $(\alpha^2-\beta^2)^2=-96 u^2 v^2$ nor the terms $\tilde{u}_1,\tilde{u}_2,\dotsc,\tilde{u}_{n-1}$. This does not hold for $p=2$, hence for this case we have $n\leq 13$. Suppose that $p\neq 2$, then $p$ does not divide $6v$. Let $\fp$ be a prime of $K=\QQ(\sqrt{-6})$ above $p$. As $(\alpha+\beta)^2$ and $\alpha\beta$ are coprime integers, and as $\alpha$, $\beta$ satisfy \eqref{eq:bprime} we see that $\gamma$, $\overline{\gamma}$ are not divisible by $\fp$. We claim the multiplicative order of the reduction of $\gamma/\overline{\gamma}$ modulo $\F_\fp$ divides $B_p$. If $-6$ is a square modulo $p$, then $\F_\fp=\F_p$ and so the multiplicative order divides $p-1=B_p$. Otherwise, $\F_\fp=\F_{p^2}$. However, $\gamma/\overline{\gamma}$ has norm $1$, and the elements of norm $1$ in $\F_{p^2}^*$ form a subgroup of order $p+1=B_p$. Thus in either case
$$
(\gamma/\overline{\gamma})^{B_p} \equiv 1 \pmod{\fp}.
$$
This implies that $p \mid \tilde{u}_{B_p}$. As $p$ is primitive divisor of $\tilde{u}_{n}$ we see that $n \le B_p$, proving the lemma.
\end{proof}

From Lemma \ref{lem:B} we know that $n\leq B_p$. For the values of $n\leq B_p$ we have to solve the equation
\begin{equation}\label{eq:eq_v_1}
\pm p^b=f_n(u,\pm 1).
\end{equation}
From the definition of $f_n(x,y)$ it holds $f_n(u,\pm 1)=f_n(u,1)$. For the rest of the paper we write $f_n(u)$ instead of $f_n(u,1)$. 

In general we do not expect solutions of \eqref{eq:eq_v_1} for big $n$ and we prove that by showing that there are no solutions of the congruence equation
\begin{equation}
\pm p^b \equiv f_n(u)\pmod s,
\end{equation}
for $s=p, p\pm 1$. This elementary criterion works for almost all cases. However, there are pairs $(p,n)$ for which it does not work. For these cases we apply Algorithm \ref{alg:non_solution_f_an} from Section \ref{sec:non_solution_f_an} which succeeds for all $n\geq 5$ apart from $(p,n) = (3109,5)$.


For $n=3,5$ the problem can be reduced to the problem of solving a certain $S$-unit equation\footnote{It can also be reduced to the problem of computing integral points on the elliptic curves $$9Y^2-2 = cX^3$$ for $c=1,p,p^2$. However, for big $p$ it is hard to compute the integral points on the elliptic curve.}. Let consider the case $n=3$. We recall that $f_3(u) = 9u^2-2$. We define $L=\QQ(\sqrt{2})$ and $\epsilon=1+\sqrt{2}$ is a generator of the free part of the unit group of $L$. Then for a prime $p$ such that $\left(\frac{2}{p}\right) = 1$ let $w$ be a generator of the prime ideal $\fp$ in\footnote{The class number of $L$ is $1$.} $L$ such that $\fp\mid p$. For an element $x\in L$ we denote by $\bar x$ its conjugate. Then we can prove that
\begin{equation}
3u-\sqrt{2} = (-1)^{b_0} \epsilon^{b_1}w^{b},
\end{equation}
for some $b_0,b_1\in\ZZ$. After conjugating and subtracting we have 
\begin{equation}\label{eq:S_unit_n3}
1 = (-1)^{b_0} \epsilon^{b_1}w^{b} (\sqrt{2})^{-3} + (-1)^{b_0} \bar\epsilon^{b_1}\bar w^{b} (-\sqrt{2})^{-3},
\end{equation}
which is an $S$-unit equation.

For $n=5$ we recall that $f_5(u)=45u^4 - 60u^2 + 4$. For this case we have to solve the equation
\begin{equation}
(15u^2-10)^2 - 80 = 5p^b.
\end{equation}
Similar to the case $n=3$ and working over $N=\QQ(\sqrt{5})$ for $p\neq 2,5$ we have
\begin{equation}
(15u^2-10) - 4\sqrt{5} = (-1)^{b_0}\epsilon^{b_1}5^{b_2}w^b,
\end{equation}
where $b_0,b_1,b_2\in\ZZ$, $\epsilon=(1+\sqrt{5})/2$ and $w$ is a generator\footnote{The class number of $N$ is $1$.} of the prime $\fp$ in $N$ above $p$.
Conjugating and subtracting we end up with the following $S$-unit equation
\begin{equation}\label{eq:S_unit_n5}
1 = (-1)^{b_0}\epsilon^{b_1}\sqrt{5}^{b_2}w^b 2^{-3}- (-1)^{b_0}\bar\epsilon^{b_1}(-\sqrt{5})^{b_2}\bar w^b 2^{-3}.
\end{equation}

Using standard and well-known algorithms (see \cite{Smartbook,Smart95,TzanakisDeWeger89,TzanakisDeWeger92}) we can find an upper bound for $b$ in \eqref{eq:S_unit_n3} and \eqref{eq:S_unit_n5}. Since we have the upper bound for $b$ we can compute $u$ looking for integer solutions of \eqref{eq:eq_v_1}. The code we have used to bound $b$ is based on author's thesis \cite{Koutsianas16} (see also \cite{Koutsianas17,AlvaradoKoutsianasMalmskogRasmussenVincentWest19}).

We have written a \texttt{Sage} script that does all the above computations. Finally, the complete list of primitive solutions of \eqref{eq:main} with $d=p^b$ and $p\leq\Bp$ are those in Table \ref{table:solutions}.

\section*{Acknowledgement}

The author is grateful to Professor John Cremona for providing access to the servers of the Number Theory Group of Warwick Mathematics Institute where all the computations took place.

\bibliographystyle{alpha}
\bibliography{Sum_of_consecutive_squares}
\end{document}